\documentclass[12pt,reqno]{amsart}

\newcommand\version{February 19, 2011}

\usepackage{amsmath,amsfonts,amsthm,amssymb,amsxtra}

\newtheorem{theorem}{Theorem}[section]
\newtheorem{proposition}[theorem]{Proposition}
\newtheorem{lemma}[theorem]{Lemma}
\newtheorem{corollary}[theorem]{Corollary}

\theoremstyle{definition}

\newtheorem{assumption}[theorem]{Assumption}

\theoremstyle{remark}

\numberwithin{equation}{section}

\newcommand{\C}{\mathbb{C}}

\renewcommand{\epsilon}{\varepsilon}

\newcommand{\R}{\mathbb{R}}

\newcommand{\Sph}{\mathbb{S}}

\DeclareMathOperator{\dist}{dist}

\DeclareMathOperator{\dom}{dom}

\DeclareMathOperator{\re}{Re}

\DeclareMathOperator{\sgn}{sgn}

\DeclareMathOperator{\tr}{Tr}

\title[HSM inequality --- \version]{Hardy-Sobolev-Maz'ya inequalities\\ for arbitrary domains}

\author{Rupert L. Frank}
\address{Rupert L. Frank, Department of Mathematics, Fine Hall,
  Princeton University, Princeton, NJ 08544, USA}
\email{rlfrank@math.princeton.edu}
\author{Michael Loss}
\address{Michael Loss, School of Mathematics, Georgia Institute of
  Technology, Atlanta, GA 30332-0160, USA}
\email{loss@math.gatech.edu}

\begin{document}

\begin{abstract}
We prove a Hardy-Sobolev-Maz'ya inequality for arbitrary domains $\Omega\subset\R^N$ with a constant depending only on the dimension $N\geq 3$. In particular, for convex domains this settles a conjecture by Filippas, Maz'ya and Tertikas. As an application we derive Hardy-Lieb-Thirring inequalities for eigenvalues of Schr\"odinger operators on domains.
\end{abstract}

\thanks{\copyright\, 2011 by the authors. This paper may be reproduced, in its entirety, for non-commercial purposes.}

\maketitle

\section{Introduction and main result}

\subsection{Hardy-Sobolev-Maz'ya inequalities}

Hardy inequalities and Sobolev inequalities bound the size of a function, measured by a (possibly weighted) $L^q$ norm, in terms of its smoothness, measured by an integral of its gradient. Maz'ya \cite{Ma} proved that for functions on the half-space $\R^N_+=\{x\in\R^N:\ x_N>0\}$, $N\geq 3$, which vanish on the boundary, the sharp version of the Hardy inequality can be combined with the Sobolev inequality into a single inequality, namely,
\begin{equation}
 \label{eq:hsmhp}
\int_{\R^N_+} \left(|\nabla u|^2 -\frac{|u|^2}{4x_N^2}\right)dx \geq \sigma_N \left( \int_{\R^N_+} |u|^\frac{2N}{N-2} \,dx \right)^{\frac{N-2}{N}}, \quad u\in C_0^\infty(\R^N_+)\,.
\end{equation}
This inequality, its generalizations to different powers of the gradient \cite{BaFiTe} and its optimal constants \cite{TeTi,BeFrLo} have attracted attention recently. Another series of papers investigates extensions of Hardy's inequality to convex domains and possible $L^2$-remainder terms \cite{BrMa,HoHoLa,FiMaTe2,AvWi}. In \cite{FiMaTe1,FiMaTe3} Filippas, Maz'ya and Tertikas found an extension of the Hardy-Sobolev-Maz'ya inequality \eqref{eq:hsmhp} to convex domains. They prove that for any convex, bounded domain $\Omega$ with $C^2$-boundary there is a constant $\sigma(\Omega)$ such that
\begin{equation}
 \label{eq:hsmconv}
\int_{\Omega} \left(|\nabla u|^2 -\frac{|u|^2}{4 \dist(x,\Omega^c)^2}\right)dx 
\geq \sigma(\Omega) \left( \int_{\Omega} |u|^\frac{2N}{N-2} \,dx \right)^{\frac{N-2}{N}} ,
\quad u\in C_0^\infty(\Omega) \,.
\end{equation}
Open problem 1 in \cite{FiMaTe1} asks whether the constant $\sigma(\Omega)$ can be chosen independently of $\Omega$. Our main result is an affirmative answer to this question.

In fact, we shall prove a more general inequality, valid for \emph{any} (not necessarily convex) domain $\Omega$. This extension is in the spirit of Davies' paper \cite{Da}, where non-negativity of the left side of \eqref{eq:hsmconv} was observed. In that paper Davies also introduced the weight function
$$
D_\Omega(x) := \left( N |\Sph^{N-1}|^{-1} \int_{\Sph^{N-1}} d_e(x)^{-2} \,de \right)^{-\frac12},
$$
where $d_e(x) := \inf\{ |t|:\ x+te \in\Omega^c \}$ for $e\in\Sph^{N-1}$, and proves that for any domain $\Omega\subsetneq\R^N$ one has
\begin{equation}
 \label{eq:hardy}
\int_\Omega |\nabla u|^2 \,dx \geq \frac14 \int_\Omega \frac{|u|^2}{D_\Omega^2} \,dx \,,
\qquad u\in C_0^\infty(\Omega) \,.
\end{equation}
The relation between \eqref{eq:hardy} and the left side of \eqref{eq:hsmconv} is that
\begin{equation}
 \label{eq:distconv}
D_\Omega(x)\leq \dist(x,\Omega^c) \quad\text{if $\Omega$ is convex.}
\end{equation}
This follows by some elementary geometric considerations.

Having introduced all the relevant notation, we are now ready to state our main result.

\begin{theorem}\label{main}
 Let $N\geq 3$. There is a constant $K_N>0$ such that for any domain $\Omega\subsetneq\R^N$ and any $u\in C_0^\infty(\Omega)$
\begin{equation}
 \label{eq:main}
\int_\Omega \left( |\nabla u|^2 - \frac{|u|^2}{4D_\Omega^2} \right) dx \geq K_N \left( \int_\Omega |u|^\frac{2N}{N-2} \,dx \right)^{\frac{N-2}{N}} \,.
\end{equation}
\end{theorem}

We emphasize that the constant $K_N$ does not depend on $\Omega$. Hence \eqref{eq:distconv} yields \eqref{eq:hsmconv} with a constant independent of $\Omega$, thereby solving the problem posed in \cite{FiMaTe1}. Our proof of \eqref{eq:main} is constructive and gives an explicit value for $K_N$. We have nothing to say, however, about its \emph{sharp} value. Is the sharp value of \eqref{eq:hsmconv} given by that in \eqref{eq:hsmhp} for any convex $\Omega$? (This is true if $\Omega$ is a ball \cite{BeFrLo}.) 

If $\Omega$ has finite measure, then \eqref{eq:main} implies by means of H\"older's inequality that
\begin{equation}
 \label{eq:hhl}
\int_\Omega \left( |\nabla u|^2 - \frac{|u|^2}{4D_\Omega^2} \right) dx 
\geq K_N \ |\Omega|^{-\frac2N}  \int_\Omega |u|^2 \,dx \,.
\end{equation}
This inequality for convex $\Omega$ and with $D_\Omega(x)$ replaced by $\dist(x,\Omega^c)$ was the original question posed in the influential paper \cite{BrMa} by Brezis and Marcus. As an answer inequality \eqref{eq:hhl} was proved in \cite{HoHoLa}; see also \cite{FiMaTe2,AvWi} for further developments.

Another application of H\"older's inequality to \eqref{eq:main} yields
\begin{equation}
 \label{eq:mainint}
\left( \int_\Omega \left( |\nabla u|^2 - \frac{|u|^2}{4D_\Omega^2} \right) dx \right)^\theta 
\left( \int_\Omega |u|^2 \,dx \right)^{1-\theta}
\geq K_N^\theta \left( \int_\Omega |u|^q \,dx \right)^{\frac 2q},
\quad \theta=\tfrac N2 \left(1-\tfrac 2q\right) \,,
\end{equation}
for all $2\leq q\leq \frac{2N}{N-2}$. It turns out that this is the correct substitute of \eqref{eq:main} in dimensions one and two.

\begin{theorem}\label{main12}
 Let $N=1$ or $N=2$ and let $2\leq q<\infty$. Then there is a constant $K_{N,\theta}>0$ such that for any domain $\Omega\subsetneq\R^N$ and any $u\in C_0^\infty(\Omega)$
\begin{equation}
 \label{eq:main12}
\left( \int_\Omega \left( |\nabla u|^2 - \frac{|u|^2}{4D_\Omega^2} \right) dx \right)^\theta 
\left( \int_\Omega |u|^2 \,dx \right)^{1-\theta}
\geq K_{N,\theta} \left( \int_\Omega |u|^q \,dx \right)^{\frac 2q},
\quad \theta=\tfrac N2 \left(1-\tfrac 2q\right) \,.
\end{equation}
\end{theorem}

Of course, \eqref{eq:main12} implies \eqref{eq:hhl} also in dimensions one and two.

We also have a generalization of Theorem \ref{main} to powers $p\geq 2$ of the gradient. The relevant weight function is now
$$
D_{\Omega,p}(x) := \left( \frac{\sqrt\pi\ \Gamma(\tfrac{N+p}2)}{\Gamma(\tfrac{p+1}2)\ \Gamma(\tfrac N2)} \ |\Sph^{N-1}|^{-1} \int_{\Sph^{N-1}} d_e(x)^{-p} \,de \right)^{-\frac1p}
$$
with $d_e(x)$ as before. We note that for $p=2$ one has $D_{\Omega,2}= D_{\Omega}$. The analogue of \eqref{eq:hardy}, which is valid for any $p>1$ and any open domain $\Omega\subsetneq\R^N$, is
\begin{equation}
 \label{eq:hardyp}
\int_\Omega |\nabla u|^p \,dx 
\geq \left(\frac{p-1}p\right)^p \int_\Omega \frac{|u|^p}{(D_{\Omega,p})^p} \,dx \,,
\qquad u\in C_0^\infty(\Omega) \,.
\end{equation}
This implies, in particular, the $L^p$ Hardy inequality for convex domains \cite{MaSo}, since by the same argument leading to \eqref{eq:distconv} one sees that
\begin{equation}
 \label{eq:distconvp}
D_{\Omega,p}(x)\leq \dist(x,\Omega^c) \quad\text{if $\Omega$ is convex.}
\end{equation}
Hardy-Sobolev inequalities for $p>2$ and \emph{smooth, convex} domains were also studied in \cite{FiMaTe3}. The following theorem extends this to \emph{arbitrary} domains.

\begin{theorem}\label{mainp}
 Let $2\leq p<N$. There is a constant $K_{N,p}$ such that for any domain $\Omega\subsetneq\R^N$ and any $u\in C_0^\infty(\Omega)$
\begin{equation}
 \label{eq:mainp}
\int_\Omega \left( |\nabla u|^p - \left(\frac{p-1}{p}\right)^p \frac{|u|^p}{(D_{\Omega,p})^p} \right) dx 
\geq K_{N,p} \left( \int_\Omega |u|^\frac{Np}{N-p} \,dx \right)^{\frac{N-p}{N}} \,.
\end{equation}
\end{theorem}

Inequalities analogous to \eqref{eq:hhl} and \eqref{eq:main12} hold for $p>2$ as well.

We shall give the details of the proofs of Theorems \ref{main}, \ref{main12} and \ref{mainp} in Section \ref{sec:proofmain}. Before summarizing the proof strategy in Subsection \ref{sec:strat}, we would first like to give an application of these theorems to spectral problems in mathematical physics.


\subsection{Hardy-Lieb-Thirring inequalities}

To motivate the following inequalities we introduce a `duality parameter' (physically: a potential) $V\in L^{N/2}(\Omega)$. If $N\geq 3$ we can infer from H\"older's inequality and \eqref{eq:main} that
\begin{align*}
 \int_\Omega \left( |\nabla u|^2 - \frac{|u|^2}{4 D_\Omega^2} + V|u|^2 \right) dx 
& \geq \int_\Omega \left( |\nabla u|^2 - \frac{|u|^2}{4 D_\Omega^2} \right) dx - \|V_-\|_{\frac N2} \|u\|_{\frac{2N}{N-2}}^2 \\ 
& \geq \int_\Omega \left( |\nabla u|^2 - \frac{|u|^2}{4 D_\Omega^2} \right) dx \left( 1 - K_N^{-1} \|V_-\|_{\frac N2} \right) \,, 
\end{align*}
where $V_-=\max\{-V,0\}$ denotes the negative part of the potential. From this we conclude that the Schr\"odinger operator
\begin{equation}
 \label{eq:schrop}
-\Delta - (2 D_\Omega)^{-2} + V
\qquad\text{in}\ L^2(\Omega)
\end{equation}
with Dirichlet boundary conditions on $\partial\Omega$ has no negative eigenvalues if the potential satisfies $\|V_-\|_{\frac N2}\leq K_N$. The following theorem improves this by saying that not only the existence of negative eigenvalues but even their \emph{total number} is controlled in terms of the $L^{N/2}$-norm of the potential. In the case of the `usual' Schr\"odinger operator $-\Delta+V$, this is the famous inequality of Cwikel, Lieb and Rozenblum. We refer to the reviews \cite{LaWe,Hu} for references, motivations and applications of this inequality. Our new result is that this inequality remains valid (possibly up to a constant), even when the positive term $(2 D_\Omega)^{-2}$ is subtracted from the Laplacian. The precise statement is

\begin{theorem}\label{hclr}
 Let $N\geq 3$. There is a constant $L_N$ such that for any domain $\Omega\subsetneq\R^N$ and any $V\in L^{N/2}(\Omega)$ the number $N(-\Delta - (2D_\Omega)^{-2} + V)$ of negative eigenvalues (including multiplicities) of the operator \eqref{eq:schrop} is bounded by
\begin{equation}
 \label{eq:hclr}
N(-\Delta - (2 D_\Omega)^{-2} + V) \leq L_N \int_\Omega V_-^{\frac N2} \,dx \,.
\end{equation}
This inequality holds with $L_N= e^{\frac N2 -1} K_N^{-\frac N2}$, where $K_N$ is the constant from \eqref{eq:main}.
\end{theorem}

For example, choosing $V=(2 D_\Omega)^{-2}-\mu$, where $\mu$ is a positive constant, we infer that the number of eigenvalues less than $\mu$ of the Dirichlet Laplacian on $\Omega$ is bounded by
$$
N(-\Delta -\mu) \leq L_N \int_\Omega \left((2 D_\Omega)^{-2}-\mu\right)_-^{\frac N2} \,dx \,.
$$
Since the latter integral can be bounded as follows,
$$
\int_\Omega \left((2 D_\Omega)^{-2}-\mu\right)_-^{\frac N2} \,dx \leq \mu^\frac N2 \ \left|\{ x\in\Omega :\ D_\Omega(x) > (4\mu)^{-\frac12} \} \right| \,,
$$
this quantifies in a uniform way the observation that because of the Dirichlet conditions most of the eigenvalues come from the bulk of $\Omega$.

It is well-known that an inequality of the form \eqref{eq:hclr} implies inequalities for moments of the negative eigenvalues $E_j$ of the operator \eqref{eq:schrop}, namely,
\begin{equation}
 \label{eq:hlt}
\sum_j |E_j|^\gamma \leq L_{N,\gamma} \int_\Omega V_-^{\gamma+\frac N2} \,dx \,.
\end{equation}
Here $\gamma>0$, and the sum runs over all (including multiplicities) negative eigenvalues of $-\Delta - (2D_\Omega)^{-2} + V$. When the term $(2D_\Omega)^{-2}$ is absent, these inequalities go back to Lieb and Thirring \cite{LiTh}. Just like \eqref{eq:mainint} is the appropriate consequence of \eqref{eq:main} that can be generalized to dimensions one and two, inequality \eqref{eq:hlt} remains valid in these dimensions.

\begin{theorem}\label{hlt12}
 Let $\gamma>1/2$ if $N=1$ and $\gamma> 0$ if $N=2$. Then there is a constant $L_{N,\gamma}$ such that for any domain $\Omega\subsetneq\R^N$ and any $V\in L^{\gamma+N/2}(\Omega)$ the negative eigenvalues $E_j$ of the operator \eqref{eq:schrop} are bounded by \eqref{eq:hlt}.
\end{theorem}

It is quite likely that \eqref{eq:hlt} remains valid for $N=1$ and $\gamma=1/2$, but we did not try to prove this. We emphasize that the main point of \eqref{eq:hlt} is its universality, being valid for any $\Omega$ and $V$ and even for small values of $\gamma$. On the other hand, in the special case of $V\equiv\ $const and for $\gamma\geq 3/2$, much more precise information about the influence of the boundary on Lieb-Thirring inequalities is available; see, e.g., the recent paper \cite{GeLaWe} and references therein.

Hardy-Lieb-Thirring inequalities of the form \eqref{eq:hlt}, but with a Hardy term becoming singular at a single point, were first derived in \cite{EkFr} and found later an application to the physical problem of stability of matter \cite{FrLiSe1}; see also \cite{Fr}. The papers \cite{FrLiSe1} and \cite{FrLiSe2} develop an approach how to deduce Lieb-Thirring-type inequalities from (a-priori weaker) Sobolev-type inequalities. Theorems \ref{hclr} and \ref{hlt12}, which appear here for the first time, were actually a main motivation for developing this abstract approach. 


\subsection{Strategy of the proof}\label{sec:strat}

In order to motivate our argument, we first review the classical proof by Gagliardo and Nirenberg of the Sobolev inequality. For simplicity we restrict ourselves to dimension $N=3$ and we want to prove that
\begin{equation}\label{eq:sob}
\int_{\R^3} |\nabla u|^2 \,dx \geq S \left( \int_{\R^3} |u|^6 \,dx \right)^{1/3} \,.
\end{equation}
The starting point of the proof is the one-dimensional Sobolev inequality
\begin{equation}
 \label{eq:sobl11}
|f(t)| \leq \frac12 \int_\R |f'|\,ds \,,
\qquad f\in C_0^\infty(\R) \,,
\end{equation}
which comes from the formula $f(t)=\frac12 \left( \int_{-\infty}^t f(s)\,ds - \int_{-\infty}^t f(s)\,ds \right)$. Now given a function $v\in C_0^\infty(\R^3)$ we apply the one-dimensional inequality to the three one-dimensional functions $t\mapsto v(t,x_2,x_3)$, $t\mapsto v(x_1,t,x_3)$ and $t\mapsto v(x_1,x_2,t)$ and we obtain
$$
|v(x)|^3 \leq \frac18\ \rho_1(x_2,x_3)\ \rho_2(x_1,x_3)\ \rho_3(x_1,x_2) \,,
$$
where
$$
\rho_1(x_2,x_3) := \int_\R \left|\frac{\partial}{\partial x_1} v(t,x_2,x_3)\right| \,dt \,,
$$
and similarly for $\rho_2$ and $\rho_3$. Then the Schwarz and the arithmetic-geometric mean inequalities imply that
$$
\int_{\R^3} |v|^{\frac32} \,dx
\leq 8^{-\frac12} \prod_{j=1}^3 \|\rho_j\|_1^{\frac12} 
\leq 8^{-\frac12} \left( \frac13\sum_{j=1}^3 \|\rho_j\|_1 \right)^\frac32
= 8^{-\frac12} \left( \frac13 \sum_{j=1}^3 \int_{\R^3} \left|\frac{\partial v}{\partial x_j} \right| \,dx \right)^\frac32 \,.
$$
This is an $L^1$ Sobolev inequality, and in order to arrive at the $L^2$ Sobolev inequality \eqref{eq:sob} we set $v=u^4$ and estimate
$$
\sum_{j=1}^3 \int_{\R^3} \left|\frac{\partial v}{\partial x_j} \right| \,dx
= 4 \sum_{j=1}^3 \int_{\R^3} \left|\frac{\partial u}{\partial x_j} \right| |u|^3 \,dx 
\leq 4 \sqrt3 \left( \int_{\R^3} |\nabla u|^2 \,dx\right)^\frac12 \left( \int_{\R^3} |u|^6 \,dx \right)^\frac12 \,.
$$
Finally dividing by $\|u\|_6^3$, we obtain \eqref{eq:sob}.

The simple observation, which is behind our proof of Theorem \ref{main}, is that one can reverse the order of the steps in the above argument. Namely, one can set already $f=g^4$ in the one-dimensional Sobolev inequality, which then becomes
\begin{equation}
 \label{eq:sobl21}
|g(t)|^4 \leq 2 \left( \int_\R |g'|^2 \,ds\right)^\frac12 \left( \int_\R |g|^6 \,dx\right)^\frac12 \,,
\qquad g\in C_0^\infty(\R) \,.
\end{equation}
Now we obtain
$$
|u(x)|^{12} \leq 8 \ \phi_1(x_2,x_3)^\frac12\ \psi_1(x_2,x_3)^\frac12\ \phi_2(x_1,x_3)^\frac12\ \psi_2(x_1,x_3)^\frac12\ \phi_3(x_1,x_2)^\frac12\ \psi_3(x_1,x_2)^\frac12 \,,
$$
with
$$
\phi_1(x_2,x_3) := \int_\R \left|\frac{\partial}{\partial x_1} u(t,x_2,x_3)\right|^2 \,dt\,,
\qquad
\psi_1(x_2,x_3) := \int_\R |u(t,x_2,x_3)|^6 \,dt \,,
$$
and similarly for the remaining functions. As before, from H\"older's inequality we get
$$
\int_{\R^3} |u|^6 \,dx
\leq 8^{\frac12} \prod_{j=1}^3 \left\|\phi_j^\frac12 \psi_j^\frac12 \right\|_1^{\frac12} \,.
$$
Next, we apply the Schwarz inequality $\|\phi_j^\frac12 \psi_j^\frac12\|_1 \leq \|\phi_j\|_1^\frac12 \|\psi_j\|_1^\frac12$, we note that $\|\psi_j\|_1=\|u\|_6^6$ and we apply the geometric-arithmetic mean inequality to conclude that
\begin{align*}
 \int_{\R^3} |u|^6 \,dx
& \leq 8^{\frac12} \|u\|_6^\frac92 \prod_{j=1}^3 \|\phi_j\|_1^\frac14
\leq 8^{\frac12} \|u\|_6^\frac92 \left( \frac13\sum_{j=1}^3 \|\phi_j\|_1 \right)^\frac34 \\
& = 8^{\frac12} \|u\|_6^\frac92 \left( \frac13 \int_{\R^3} |\nabla u|^2 \,dx \right)^\frac34 \,,
\end{align*}
which is our desired goal \eqref{eq:sob}.

The upshot of this discussion is that in order to arrive at the $L^2$ Sobolev inequality \eqref{eq:sob} (which is weaker than the $L^1$ Sobolev inequality) we only need the one-dimensional $L^2$ Sobolev inequality \eqref{eq:sobl21}, and not the one-dimensional $L^1$ Sobolev inequality \eqref{eq:sobl11}. This simple observation is of relevance for us because there is not even a Hardy inequality, i.e., the inequality
$$
\int_{-1}^1 |f'(x)| dx \ge C \int_{-1}^1 \frac{|f(x)|}{1-|x|} dx
$$
is false no matter how small the constant $C$.  However, and this is our technical key result (Proposition \ref{key}), we can prove a one-dimensional $L^2$ Sobolev-type inequality with a Hardy term! In fact such inequalities hold for all $p \ge 2$ (Proposition \ref{keyp}). For $1 < p <2$ there is a Hardy inequality, however, it is not known whether any version of Proposition \ref{keyp} might hold for $1 < p< 2$. Once such an inequality is established, the rest of the argument outlined above yields Hardy-Sobolev-Mazy'a inequalities also for $1< p< 2$.

\medskip
\noindent
{\bf Acknowledgment:} The work of M.L. is partially funded by NSF grant DMS--901304.


\section{Proof of Theorems \ref{main}, \ref{main12} and \ref{mainp}}\label{sec:proofmain}

\subsection{A one-dimensional inequality}

The following inequality is the key for proving Theorem \ref{main}.

\begin{proposition}\label{key}
 Let $q\geq2$. There is a constant $C_q$ such that for every $f\in C_0^\infty(-1,1)$ and for every $t\in [-1,1]$ one has
\begin{equation}
 \label{eq:key}
|f(t)|^{q+2} \leq C_q \int_{-1}^1 \left(|f'|^2 -\frac{|f|^2}{4(1-|s|)^2} \right)\,ds 
\ \int_{-1}^1 |f|^q \,ds \, ,
\end{equation}
with $C_q \le (q+2)^2$.
\end{proposition}

\begin{proof}
We begin by noting that if we write $f(t)=\sqrt{1-|t|} \ g(t)$, then what we have to show is that
$$
|g(t)|^{q+2} \leq  (q+2)^2 (1-|t|)^{-\frac{q+2}{2}} \left( \int_{-1}^1 |g'|^2 (1-|s|)\,ds + |g(0)|^2 \right) \int_{-1}^1 |g|^q (1-|s|)^\frac q2 \,ds \,.
$$
By symmetry it suffices to prove this for $t\in[0,1]$ only. Now for such $t$ we can use the fact that $(1-t)^{\frac{q+2}{4}}$ is decreasing and we find that
\begin{align*}
|g(t)|^\frac{q+2}{2} - |g(0)|^\frac{q+2}{2} 
& \leq \tfrac{q+2}{2} \int_0^t |g|^\frac q2 |g'| \,ds \\ 
& \leq \tfrac{q+2}{2} (1-t)^{-\frac{q+2}{4}} \int_0^1 |g|^\frac q2 |g'| (1-s)^{\frac{q+2}{4}}\,ds \\
& \leq \tfrac{q+2}{2} (1-t)^{-\frac{q+2}{4}} \left( \int_0^1 |g'|^2 (1-s)\,ds \right)^{\frac12} \left( \int_0^1 |g|^q (1-s)^{\frac q2} \,ds \right)^{\frac12} \,.
\end{align*}
Thus it remains to show that
\begin{equation}\label{eq:goal}
|g(0)|^{q+2} \leq \frac{(q+2)^2}{4} \left( \int_{-1}^1 |g'|^2 (1-|s|)\,ds + |g(0)|^2 \right) \int_{-1}^1 |g|^q (1-|s|)^{\frac q2} \,ds \,.
\end{equation}
Of course, it suffices to show this if $g$ is non-negative, which is what we will assume henceforth. Let $\alpha$ be a parameter (to be specified later). Since $(1-s)^\alpha g(s)^{(q+2)/2}$ vanishes near $s=1$ we can write
\begin{align*}
|g(0)|^\frac{q+2}{2} 
&  = - \int_0^1 \left(\tfrac{q+2}{2} (1-s)^\alpha g(s)^\frac{q}{2} g'(s) -\alpha (1-s)^{\alpha-1} g(s)^\frac{q+2}{2} \right) \,ds \\
&  = - \int_0^1 g(s)^\frac{q}{2} (1-s)^\frac{q}{4} \left( \tfrac{q+2}{2} (1-s)^{\alpha-\frac{q}{4}}  g'(s) -
\alpha g(s)(1-s)^{\alpha-\frac{q+4}{4}} \right)\, ds \, .
\end{align*}
Using the Schwarz inequality we find
\begin{eqnarray*}
|g(0)|^\frac{q+2}{2} \leq \left( \int_0^1 g(s)^{q} (1-s)^\frac{q}{2} ds\right)^{1/2} T^{1/2}
\end{eqnarray*}
with
\begin{align*}
T & = \int_0^1 \left( \tfrac{(q+2)^2}{4} (1-s)^{2\alpha-\frac{q}2}  g'^2
- (q+2) \alpha (1-s)^{2\alpha-\frac{q+2}2}  g g' + \alpha^2 g^2 (1-s)^{2\alpha-\frac{q+4}{2}} 
\right) ds \\
& = \int_0^1 \left( \tfrac{(q+2)^2}{4} (1-s)^{2\alpha-\frac{q}2}  g'^2
- \tfrac{q+2}2 \alpha (1-s)^{2\alpha-\frac{q+2}2}  (g^2)' + \alpha^2 g^2 (1-s)^{2\alpha-\frac{q+4}{2}} 
\right) ds \\
& = \int_0^1 \left( \tfrac{(q+2)^2}{4} (1-s)^{2\alpha-\frac{q}2}  g'^2
+ \alpha \left( \alpha - \tfrac{q+2}2 \left(2\alpha-\tfrac{q+2}2\right)\right) (1-s)^{2\alpha-\frac{q+4}{2}}                  g^2 \right) ds
+ \alpha \tfrac{q+2}2 g(0)^2 \,.
\end{align*}
Now we pick $\alpha$ such that $\alpha - \frac{q+2}2 (2\alpha-\frac{q+2}2) = 0$, which leads to
$$
\alpha = \tfrac{(q+2)^2}{4(q+1)} \,.
$$
Hence we have
$$
|g(0)|^\frac{q+2}{2} \leq \tfrac{q+2}{2} \left( \int_0^1  g^{q} (1-s)^\frac{q}{2} ds\right)^{1/2} 
\left( \int_0^1 g'(s)^2 (1-s)^{\frac{(q+2)^2}{2(q+1)}-\frac{q}2} ds
+ \tfrac{q+2}{2(q+1)} g(0)^2 \right)^{1/2} ,
$$
which, since $\frac{(q+2)^2}{2(q+1)}-\frac{q}2 \geq 1$ and $\tfrac{q+2}{2(q+1)} \leq 1$, is bounded above by
$$
\tfrac{q+2}{2} \left( \int_0^1  (1-s)^\frac{q}{2} g(s)^{q} ds\right)^{1/2} 
\left( \int_0^1 g'(s)^2 (1-s) ds + g(0)^2 \right)^{1/2} \,.
$$
This proves the claimed inequality \eqref{eq:goal}.
\end{proof}

\begin{corollary}\label{keycor}
Let $q\geq2$. Then, with the same constant $C_q$ as in \eqref{eq:key}, one has for every open set $\Omega\subsetneq\R$ and for every $f\in C_0^\infty(\Omega)$
\begin{equation}
 \label{eq:keycor}
\sup_{t\in\Omega} |f(t)|^{q+2} \leq C_q \int_\Omega \left(|f'|^2 -\frac{|f|^2}{4\dist(t,\Omega^c)^2} \right)\,dt
\ \int_\Omega |f|^q \,dt \,.
\end{equation}
\end{corollary}

\begin{proof}
 First, if $\Omega$ is an interval, then \eqref{eq:keycor} follows from \eqref{eq:key} by a translation and a dilation. Now the extension to arbitrary open sets (that is, countable unions of disjoint intervals) is straightforward.
\end{proof}

The following inequality will be needed to deal with the two dimensional case.

\begin{corollary}\label{keycor2}
Let $q\geq4$. Then, with the same constant $C_q$ as in \eqref{eq:key}, one has for every open set $\Omega\subsetneq\R$ and for every $f\in C_0^\infty(\Omega)$
\begin{equation*}
\sup_{t\in\Omega} |f(t)|^{q} \leq C_{q-2} \int_\Omega \left(|f'|^2 -\frac{|f|^2}{4\dist(t,\Omega^c)^2} \right)\,dt 
\left( \int_\Omega |f|^2 \,dt \right)^{\frac{2}{q-2}}
\left( \int_\Omega |f|^q \,dt \right)^{\frac{q-4}{q-2}}  \,.
\end{equation*}
\end{corollary}

\begin{proof}
 We apply Corollary \ref{keycor} with $q$ replaced by $q-2$ and we estimate $\|f\|_{q-2}^{q-2}$ using H\"older's inequality.
\end{proof}


\subsection{The inequality in dimensions $N\geq 3$}

In order to pass from the one-di\-men\-sio\-nal inequality of Corollary \ref{keycor} to Theorem \ref{main} we use the well-known argument of Gagliardo and Nirenberg. We shall use the following notation for $x\in\R^N$ and $1\leq j\leq N$,
$$
\tilde x_j=(x_1,\ldots,x_{j-1},x_{j+1},\ldots,x_N)\in\R^{N-1} \,.
$$
Then one has

\begin{lemma}\label{gn}
 Let $N\geq 2$ and let $f_1,\ldots,f_N\in L^{N-1}(\R^N)$. Then the function $f(x):= f_1(\tilde x_1)\cdots f_N(\tilde x_N)$ belongs to $L^1(\R^N)$ and
$$
\|f\|_{L^1(\R^N)} \leq \prod_{j=1}^N \|f_j\|_{L^{N-1}(\R^N)} \,.
$$
\end{lemma}

The easy proof, based on H\"older's inequality, can be found for instance in \cite{Br}.

With Lemma \ref{gn} at hand we now are ready to give the \emph{proof of Theorem \ref{main}}. Let $e_1,\ldots,e_N$ be the standard unit vectors in $\R^N$. For a given domain $\Omega\subsetneq\R^N$ we write $d_j$ instead of $d_{e_j}$, that is,
$$
d_j(x) = \inf\{ |t|:\ x+te_j \in\Omega^c \} \,.
$$
Now if $u\in C_0^\infty(\Omega)$, then Corollary \eqref{eq:keycor} yields
$$
|u(x)| \leq C_q \left( g_j(\tilde x_j) h_j(\tilde x_j) \right)^{\frac{N-2}{4(N-1)}}
$$
for any $1\leq j\leq N$, where
$$
g_j(\tilde x_j) := \int_\R \left( \left|\frac{\partial u}{\partial x_j} (x)\right|^2 - \frac{|u(x)|^2}{4 d_j(x)^2} \right)dx_j
\quad\text{and}\quad
h_j(\tilde x_j) := \int_\R |u(x)|^q \,dx_j
$$
with $q=\frac{2N}{N-2}$. Thus
$$
|u(x)|^N \leq C_q^N \prod_{j=1}^N \left( g_j(\tilde x_j) h_j(\tilde x_j) \right)^{\frac{N-2}{4(N-1)}} \,,
$$
or, what is the same,
$$
|u(x)|^q \leq C_q^q \prod_{j=1}^N \left( g_j(\tilde x_j) h_j(\tilde x_j) \right)^{\frac 1{2(N-1)}} \,.
$$
From Lemma \ref{gn} we infer that
$$
\int_{\R^N} |u(x)|^q \,dx \leq C_q^q \prod_{j=1}^N \left( \int_{\R^{N-1}} \sqrt{g_j(y) h_j(y)} \,dy \right)^{\frac 1{N-1}} \,.
$$
Now we use the fact that
$$
\|h_j\|_{L^1(\R^{N-1})} = \| u\|_{L^q(\R^N)}^q
\qquad\text{for every}\ j=1,\ldots, N\,,
$$
and derive from the Schwarz and the arithmetic-geometric mean inequality that
\begin{align*}
 \prod_{j=1}^N \int_{\R^{N-1}} \sqrt{g_j(y) h_j(y)} \,dy 
& \leq \prod_{j=1}^N \|g_j\|_1^{\frac12} \|h_j\|_1^{\frac 12}
= \| u\|_q^{\frac{Nq}2} \prod_{j=1}^N \|g_j\|_1^{\frac12} \\
& \leq \| u\|_q^{\frac{Nq}2} \left( N^{-1} \sum_{j=1}^N \|g_j\|_1 \right)^{\frac N2} .
\end{align*}
To summarize, we have shown that
$$
\int_{\R^N} |u(x)|^q \,dx \leq C_q^q \| u\|_q^{\frac{Nq}{2(N-1)}}
 \left( N^{-1} \sum_{j=1}^N \|g_j\|_1 \right)^{\frac N{2(N-1)}} \,,
$$
that is,
\begin{align*}
\left( \int_{\R^N} |u(x)|^q \,dx \right)^{\frac 2q}
& \leq C_q^{\frac{4(N-1)}{N-2}} N^{-1} \sum_{j=1}^N \|g_j\|_1 \\
& = C_q^{\frac{4(N-1)}{N-2}} N^{-1} \int_\Omega \left( |\nabla u|^2 - \frac14 \sum_{j=1}^N \frac{|u|^2}{d_j^2} \right)dx \,.
\end{align*}
Finally, as in \cite{Da}, we average over all choices of the coordinate system and obtain the inequality claimed in Theorem \ref{main}.
\qed


\subsection{The inequality in dimensions one and two}

Next, we prove Theorem \ref{main12}.

\emph{The case $N=1$.} We bound $\|f\|_q^q \leq \|f\|_\infty^{q-2} \|f\|_2^2$ and apply Corollary \ref{keycor} to obtain
$$
\int |f|^q \,dt \leq \|f\|_\infty^{q-2} \|f\|_2^2 \leq C_q^\frac{q-2}{q+2} \left( \int_\Omega \left(|f'|^2 -\frac{|f|^2}{4\dist(t,\Omega^c)^2} \right)\,dt \right)^\frac{q-2}{q+2} \|f\|_q^\frac{q(q-2)}{q+2} \|f\|_2^2 \,.
$$
This is the inequality claimed in Theorem \ref{main12}.

\emph{The case $N=2$.} Here we proceed similarly to the case $N\geq 3$. We first observe that by H\"older's inequality it suffices to prove the inequality only for large $q$, say $q\geq 4$. For such $q$ we can apply Corollary \ref{keycor2} and obtain
$$
|u(x)|^q \leq C_{q-2}^q \prod_{j=1}^2 \left( g_j(\tilde x_j)^\frac12 h_j(\tilde x_j)^\frac{q-4}{2(q-2)}  k_j(\tilde x_j)^\frac1{q-2} \right) \,,
$$
where $g_j$ and $h_j$ are defined as before and where
$$
k_j(\tilde x_j) := \int_\R |u(x)|^2 \,dx_j \,.
$$
We integrate this inequality over $\R^2$ (note that Lemma \ref{gn} is trivial for $N=2$) and bound using H\"older's inequality,
\begin{align*}
\int_{\R^2} |u(x)|^q \,dx 
& \leq C_{q-2}^q \prod_{j=1}^2 \int_{\R} g_j(y)^\frac12 h_j(y)^\frac{q-4}{2(q-2)}  k_j(y)^\frac1{q-2} \,dy \\
& \leq C_{q-2}^q \prod_{j=1}^2 \|g_j\|_1^\frac12 \|h_j\|_1^\frac{q-4}{2(q-2)} \|k_j\|_1^\frac1{q-2} \\
& = C_{q-2}^q \|u\|_q^\frac{q(q-4)}{q-2} \|u\|_2^\frac{4}{q-2} \prod_{j=1}^2 \|g_j\|_1^\frac12 \,.
\end{align*}
The claimed inequality now follows as before by the arithmetic-geometric mean inequality for $\prod_j \|g_j\|_1$ and by averaging over all coordinate systems.
\qed


\subsection{The case $p>2$}

The analogue of Proposition \ref{key} is

\begin{proposition}\label{keyp}
 Let $q\geq p\geq 2$. There is a constant $C_{p,q}$ such that for every $f\in C_0^\infty(-1,1)$ and for every $t\in [-1,1]$ one has
\begin{equation}
 \label{eq:keyp}
|f(t)|^{q(p-1)+p} \leq C_{p,q} \int_{-1}^1 \left(|f'|^p - \left(\frac{p-1}{p}\right)^p \frac{|f|^p}{(1-|s|)^p} \right)\,ds 
\ \left( \int_{-1}^1 |f|^q \,ds \right)^{p-1} \,.
\end{equation}
\end{proposition}

Given this inequality, Theorem \ref{mainp} follows again by the Gagliardo-Nirenberg argument as in the previous subsections, but now with $q=Np/(N-p)$. We omit the details, we only point out that the constant in the definition of $D_{\Omega,p}$ appears through the evaluation of the integral
$$
|\Sph^{N-1}|^{-1}  \int_{\Sph^{N-1}} |a\cdot e|^p \,de
= \frac{\Gamma(\tfrac{p+1}2)\ \Gamma(\tfrac N2)}{\sqrt\pi\ \Gamma(\tfrac{N+p}2)} \ |a|^p 
$$
for $a\in\R^N$.

\begin{proof}
We shall use that $|a+b|^p \geq |a|^p+ p|a|^{p-2} \re\overline a b +c_p |b|^p$ for all $a,b\in\C$ and some explicit $c_p$, see \cite{FrSe}. (Here we use that $p\geq 2$.) Hence if we write $f(t)=(1-|t|)^{(p-1)/p} \ g(t)$, then
\begin{align*}
& \int_{-1}^1 \left(|f'|^p - \left(\tfrac{p-1}{p}\right)^p \tfrac{|f|^p}{(1-|s|)^p} \right)\,ds \\
& = \int_{-1}^1 \left( \left|(1-|s|)^{(p-1)/p}g' - \tfrac{p-1}p(\sgn s)(1-|s|)^{-1/p}g\right|^p - \left(\tfrac{p-1}{p}\right)^p \frac{|g|^p}{1-|s|} \right)\,ds \\
& \geq \int_{-1}^1 \left( -p \left(\tfrac{p-1}{p}\right)^{p-1} (\sgn s) |g|^{p-2} \re \overline g g' + c_p (1-|s|)^{p-1}|g'|^p \right)\,ds \\
& = 2 \left(\tfrac{p-1}{p}\right)^{p-1} |g(0)|^p + c_p \int_{-1}^1 (1-|s|)^{p-1}|g'|^p \,ds \,.
\end{align*}
Thus it is enough to show that
\begin{align*}
|g(t)|^{q(p-1)+p} \leq C (1-|t|)^{-\frac{(p-1)(p+q(p-1))}{p}} & \left( \int_{-1}^1 (1-|s|)^{p-1}|g'|^p \,ds + d |g(0)|^p \right) \\
& \times \left( \int_{-1}^1 |g|^q (1-|s|)^{\frac{q(p-1)}{p}} \,ds \right)^{p-1} \,.
\end{align*}
where  $d = 2c_p^{-1} (\frac{p-1}{p})^{p-1}$.
By symmetry it suffices to prove this for $t\in[0,1]$ only. Now for such $t$ we can use the fact that $(1-t)^{\frac{(p-1)(q(p-1)+p)}{p^2}}$ is decreasing and we find that
\begin{align*}
& |g(t)|^\frac{q(p-1)+p}{p} - |g(0)|^\frac{q(p-1)+p}{p} 
\leq \tfrac{q(p-1)+p}{p} \int_0^t |g|^\frac{q(p-1)}p |g'| \,ds \\ 
& \quad \leq \tfrac{q(p-1)+p}{p} (1-t)^{-\frac{(p-1)(p+q(p-1))}{p^2}} \int_0^1 |g|^\frac{q(p-1)}p |g'| (1-s)^{\frac{(p-1)(p+q(p-1))}{p^2}}\,ds \\
& \quad \leq \tfrac{q(p-1)+p}{p} (1-t)^{-\frac{(p-1)(p+q(p-1))}{p^2}}\!\! \left( \int_0^1 |g'|^p (1-s)^{p-1}ds \right)^{\!\!\frac1p}\! \left( \int_0^1 |g|^q (1-s)^{\frac{q(p-1)}p} ds \right)^{\!\!\frac{p-1}p} .
\end{align*}
Thus it remains to show that
\begin{equation}\label{eq:goalgeneral}
|g(0)|^{q(p-1)+p} \leq C \left( \int_{-1}^1 |g'|^p (1-|t|)^{p-1}\,dt + d|g(0)|^p \right) \left(\int_{-1}^1 |g|^q (1-|t|)^{\frac{q(p-1)}p} \,dt \right)^{p-1} \,.
\end{equation}
In order to prove this, we choose a free parameter $T\in (0,1)$ and a Lipschitz function $\chi$ with $0\leq \chi\leq 1$, $\chi(0)=1$ and $\chi(t)=0$ for $|t|\in[T,1]$. We put
$$
L := \left( \int_{-1}^1 |\chi'|^\frac{pq}{q-p} ds \right)^{\frac{pq}{q-p}} \,.
$$
Now we choose another parameter $A$ (which will be fixed later depending on $T$ and $L$) and distinguish two cases according to whether
\begin{equation}\label{eq:case1}
|g(0)|^q \leq A^\frac{p}{p-1} \int_{-1}^1 |g|^q (1-|s|)^\frac{q(p-1)}{p} \,ds
\end{equation}
or not. In the first case, we can trivially estimate
\begin{align*}
|g(0)|^\frac{q(p-1)+p}{p} & \leq A\ |g(0)|\ \left( \int_{-1}^1 |g|^q (1-|t|)^{\frac{q(p-1)}p}\,dt \right)^{\frac{p-1}p} \\
& \leq Ad^{-\frac1p} \left( \int_{-1}^1 |g'|^p (1-|t|)^{p-1}\,dt + d|g(0)|^p \right)^{\!\!\frac1p} \!\left( \int_{-1}^1 |g|^q (1-|t|)^{\frac{q(p-1)}p}\,dt \right)^{\!\!\!\frac{p-1}p}
\end{align*}
and we have arrived at our goal \eqref{eq:goalgeneral}. Now assume that the opposite inequality in \eqref{eq:case1} holds. We define $g_0:=\chi g$ and estimate this function similarly as above. Indeed, since $g_0(T)=g_0(-T)=0$,
\begin{align*}\label{eq:g1bound}
|g_0(0)|^\frac{q(p-1)+p}{p} & \leq \tfrac{q(p-1)+p}{2p} \int_{-T}^T |g_0|^{\frac{q(p-1)}p} |g_0'| \,ds \\
& \leq \tfrac{q(p-1)+p}{2p} (1-T)^{-\frac{(p-1)(p+q(p-1))}{p^2}} \int_{-1}^1 |g_0|^\frac{q(p-1)}p |g_0'| (1-|s|)^{\frac{(p-1)(p+q(p-1))}{p^2}}\,ds \\
& \leq \tfrac{q(p-1)+p}{2p} (1-T)^{-\frac{(p-1)(p+q(p-1))}{p^2}} \left( \int_{-1}^1 |g_0'|^p (1-|s|)^{p-1}\,ds \right)^{\frac1p}\\
& \qquad\qquad\qquad\qquad\qquad\qquad\qquad \times\left( \int_{-1}^1 |g_0|^{q} (1-|s|)^{\frac{q(p-1)}p}\,ds \right)^{\frac{p-1}p} \,.
\end{align*}
In order to again arrive at \eqref{eq:goalgeneral} we recall that $g_0(0)=g(0)$ and that one has
$$
\int_{-1}^1 |g_0|^{q} (1-|s|)^{\frac{q(p-1)}p}\,ds \leq \int_{-1}^1 |g|^{q} (1-|s|)^{\frac{q(p-1)}p}\,ds \,.
$$
Finally, we estimate the term involving $g_0'$ by means of the triangle inequality
\begin{align*}
\left( \int_{-1}^1 |g_0'|^p (1-|s|)^{p-1}\,ds \right)^\frac1p
& \leq \left( \int_{-1}^1 |g'|^p \chi^p (1-|s|)^{p-1}\,ds \right)^\frac1p\\
& \qquad\qquad + \left( \int_{-1}^1 |g|^p |\chi'|^p (1-|s|)^{p-1}\,ds \right)^\frac1p \\
& \leq \left( \int_{-1}^1 |g'|^p (1-|s|)^{p-1}\,ds \right)^\frac1p \\
& \qquad\qquad + L \left( \int_{-1}^1 |g|^q (1-|s|)^\frac{q(p-1)}p \,ds \right)^\frac 1q \\
& \leq \left( \int_{-1}^1 |g'|^p (1-|s|)^{p-1}\,ds \right)^\frac1p 
+ L A^{-\frac{p}{q(p-1)}} |g(0)| \\
& \leq 2^\frac{p-1}p \left( \int_{-1}^1 |g'|^p (1-|s|)^{p-1}\,ds + L^p A^{-\frac{p^2}{q(p-1)}} |g(0)|^p \right)^\frac1p \,,
\end{align*}
where in the next to last step we used the inequality opposite to \eqref{eq:case1}. Thus choosing $A$ large enough so that $L^p A^{-\frac{p^2}{q(p-1)}}\leq d$ we arrive again at \eqref{eq:goalgeneral}.
\end{proof}


\section{Proofs of Theorems \ref{hclr} and \ref{hlt12}}

\subsection{Equivalence of Sobolev and Lieb-Thirring inequalities}

We shall deduce Theorem \ref{hclr} from Theorem \ref{main} by applying the abstract approach developed in \cite{FrLiSe2}. Let us briefly summarize the main result of \cite{FrLiSe2}. Let $X$ be a sigma-finite measure space and let $t$ be a closed, non-negative quadratic form in $L^2(X)$ with domain $\dom t$. We assume the following

\begin{assumption}[Generalized Beurling-Deny conditions]\label{ass:bd}
$\phantom{x}$
\begin{enumerate}
 \item[(a)] 
if $u,v\in\mathrm{dom}\, t$ are real-valued, then $t[u+iv]=t[u]+t[v]$,
\item[(b)]
if $u\in\mathrm{dom}\, t$ is real-valued, then $|u|\in\mathrm{dom}\, t$ and $t[|u|]\leq t[u]$,
\item[(c)]
there is a measurable, a.e. positive function $\omega$ such that if $u\in\mathrm{dom}\, t$ is non-negative, then $\min(u,\omega)\in\mathrm{dom}\, t$ and $t[\min(u,\omega)]\leq t[u]$. Moreover, there is a form core $\mathcal Q$ of $t$ such that $\omega^{-1}\mathcal Q$ is dense in $L^2(X, \omega^{2\kappa/(\kappa-1)})$.
\end{enumerate}
\end{assumption}

The main result from \cite{FrLiSe2} concerns the equivalence of an estimate on the number $N(T-V)$ of negative eigenvalues of the operator $T+V$, taking multiplicities into account, and the validity of a Sobolev inequality.

\begin{theorem}\label{clrweighted}
 Under Assumption \ref{ass:bd} for some $\kappa>1$ the following are equivalent:
\begin{enumerate}
 \item[(i)]\label{it:sobolevweighted} $T$ satisfies a Sobolev inequality with exponent $q=2\kappa/(\kappa-1)$, that is, there is a constant $S>0$ such that for all $u\in\mathrm{dom}\, t$,
\begin{equation}\label{eq:sobolevweighted}
 t[u] \geq S \left( \int_X |u|^q \,dx \right)^{2/q} \,.
\end{equation}
\item[(ii)]\label{it:clrweighted} $T$ satisfies a CLR inequality with exponent $\kappa$, that is, there is a constant $L>0$ such that for all $0\geq V\in L^\kappa(X)$,
\begin{equation}\label{eq:clrweighted}
 N(T+V) \leq L \int_X V_-^{\kappa} \,dx \,.
\end{equation}
\end{enumerate}
The respective constants are bounded in terms of each other according to
\begin{equation}\label{eq:CLRconst}
S^{-\kappa} \leq L  \leq e^{\kappa-1} S^{-\kappa} \,.
\end{equation}
\end{theorem}

This theorem has is origins in the Li-Yau proof of the CLR inequality \cite{LiYa} and we refer to \cite{FrLiSe2} for further references.

We now show how to apply this theorem in order to deduce a weak form of Theorem~\ref{hclr} for \emph{convex} domains $\Omega$, namely,
\begin{equation}
 \label{eq:hclrconv}
N(-\Delta - (2\dist(x,\Omega^c))^{-2} + V) \leq L_N \int_\Omega V_-^{\frac N2} \,dx \,.
\end{equation}
The general case is, unfortunately, more complicated and will be dealt with in the following subsection. Obviously, the (closure of the) quadratic form
$$
t[u] := \int_\Omega \left( |\nabla u|^2 - \frac{|u|^2}{4\dist(x,\Omega^c)^2} \right) dx \,, \qquad u\in C_0^\infty(\Omega)\,,
$$
in the Hilbert space $L^2(\Omega)$ satisfies conditions (a) and (b) above. Moreover, from the identity
\begin{equation}
 \label{eq:gsr}
t[v\omega] = \int_\Omega \left( |\nabla v|^2 - \frac{\Delta \dist(x,\Omega^c)}{2\ \dist(x,\Omega^c)} |v|^2 \right) \dist(x,\Omega^c)\,dx
\end{equation}
with $\omega:=\sqrt{\dist(x,\Omega^c)}$ and from the fact that $\Delta \dist(x,\Omega^c)\leq 0$ as a distribution, we easily deduce that (c) is satisfied as well. Our Hardy-Sobolev-Maz'ya inequality \eqref{eq:main} and \eqref{eq:distconv} show that (i) in Theorem \ref{clrweighted} is valid, and therefore lead to \eqref{eq:hclrconv}.


\subsection{Proof of Theorem \ref{hclr}. The general case.}

The problem with the more general inequality involving the function $D_\Omega$ is that we do not know how to verify Assumption (c). In particular, we are not aware of a useful analogue of \eqref{eq:gsr}. We can use, however, the following remark (see the end of Section 4.1 in \cite{FrLiSe2}):

\emph{Theorem \ref{clrweighted} remains valid if \emph{(c)} is replaced by the following condition.}

\begin{enumerate}
 \item[(d)]
 For every a.e. positive function $W\in L^1(X)\cap L^\infty(X)$ consider the self-adjoint, non-negative operator $\Upsilon$ in $L^2(X,Wdx)$ associated to the quadratic form $t[u]$. Then $\exp(-\beta \Upsilon)$ is an integral operator in $L^2(X,Wdx)$ for every $\beta>0$.
\end{enumerate}

We are going to prove Theorem \ref{hclr} using (d) instead of (c). For technical reasons we have to work with regularizations defined by
\begin{equation}
 \label{eq:teps}
t_\epsilon[u] := \int_\Omega \left( |\nabla u|^2 - (1-\epsilon) \frac{|u|^2}{4D_\Omega^2} \right) dx \,,
\qquad u\in H^1_0(\Omega)\,,
\end{equation}
with $\epsilon\in(0,1]$. As before, $t_\epsilon$ satisfies (a), (b) and (i) with a constant which can be chosen independently of $\epsilon$. (Namely, $S=K_N$ from Theorem \ref{main}.) Hence if we can verify (d) for any $\epsilon\in(0,1]$, Theorem \ref{clrweighted} yields the inequality $N(T_\epsilon+V) \leq L \int_\Omega V_-^{N/2} \,dx$ for the operators $T_\epsilon$ associated to $t_\epsilon$. Here $L$ is a constant independent of $\epsilon$. Similarly as in \cite{FrLiSe1} one can show that $T_\epsilon+V\searrow T_0+V$ in strong resolvent sense. Therefore, if $P_\epsilon$ and $P_0$ are the spectral projectors of $T_\epsilon+V$ and $T_0+V$ corresponding to $(-\infty,0)$, then $P_\epsilon\to P_0$ strongly, and by Fatou's lemma for traces $N(T_0+V)=\tr P_0 \leq \liminf_{\epsilon\to 0} \tr P_\epsilon = \liminf_{\epsilon\to 0} N(T_\epsilon+V) \leq L \int_\Omega V_-^{N/2} \,dx$, as claimed.

Thus, to complete the proof of Theorem \ref{hclr} we need to verify that $\exp(-\beta \Upsilon_\epsilon)$ is an integral operator in $L^2(\Omega,Wdx)$. Here $W$ is a given, a.e. positive function in $L^1(\Omega)\cap L^\infty(\Omega)$, $\beta>0$ is a constant and $\Upsilon_\epsilon$ is the self-adjoint, non-negative operator in $L^2(\Omega,Wdx)$ associated with the quadratic form $t_\epsilon$ from \eqref{eq:teps}. We note that $\Upsilon_\epsilon u = W^{-1} (-\Delta-(1-\epsilon)(2D_\Omega)^{-2})u$ for $u\in C_0^\infty(\Omega)$. Since the coefficients of this operator are not smooth, the existence of an integral kernel is not completely standard and we include a short proof.

We claim that $\exp(-\beta \Upsilon_\epsilon)$ is, in fact, a Hilbert-Schmidt operator in the space $L^2(\Omega,Wdx)$. Via the unitary mapping $L^2(\Omega,Wdx)\ni u\mapsto \sqrt W u\in L^2(\Omega,dx)$, this is equivalent to saying that the operator $\exp(-\beta H_\epsilon)$ in the space $L^2(\Omega,dx)$ is a Hilbert-Schmidt operator, where $H_\epsilon:=W^{-\frac12} (-\Delta-(1-\epsilon)(2D_\Omega)^{-2}) W^{-\frac12}$. This, in turn, will follow if we can prove that the eigenvalues $e_j$ of the operator $H_\epsilon$ satisfy a bound of the form $e_j \geq C j^{2/N}$, where the constant $C$ may depend on $W$ and $\epsilon$, but is independent of $j$. In Lemma \ref{clr} below we show that a bound of this form is true for the operator $W^{-\frac12} (-\Delta) W^{-\frac12}$, where $-\Delta$ is the Dirichlet Laplacian on $\Omega$. Now Davies' Hardy inequality \eqref{eq:hardy} implies that
$$
t_\epsilon[u] \geq \epsilon \int_\Omega |\nabla u|^2\,dx \,,
\qquad u\in H^1_0(\Omega)\,,
$$
and therefore $H_\epsilon\geq \epsilon W^{-\frac12} (-\Delta) W^{-\frac12}$ in the sense of quadratic forms. The inequality for $e_j$ now follows from the variational principle. This completes the proof of Theorem~\ref{hclr}.
\qed

\medskip

In the previous proof we used a lower bound on the $j$-th eigenvalue of the operator $W^{-\frac12} (-\Delta) W^{-\frac12}$ in $L^2(\Omega,dx)$, where $W$ is an a.e. positive function in $L^1(\Omega)\cap L^\infty(\Omega)$. For later purposes we state a similar bound also in dimensions one and two.

\begin{lemma}\label{clr}
 Let $\Omega\subset\R^N$ and let $\tau>0$ if $N=1,2$ and $\tau=0$ if $N\geq 3$. Let $(\mu_j)$ be the increasing sequence of eigenvalues (counting multiplicities) of the operator $W^{-\frac12} (-\Delta+\tau) W^{-\frac12}$ in $L^2(\Omega)$. Then
$$
\mu_j \geq C_N \ \|W\|_{\frac N2}^{-1} \ j^{\frac 2N}
\qquad\text{if}\ N\geq 3\,,
$$
and
$$
\mu_j \geq C_{N,p}\ \tau^{1-\frac{N}{2p}}\ \|W\|_{p}^{-1} \ j^{\frac 1p}
\qquad\text{if}\ N=1,2\,,
$$
where $p\geq 1$ if $N=1$ and $p>1$ if $N=2$.
\end{lemma}

These bounds are not new. In the following proof we shall make use of the observation that $W^{-\frac12} (-\Delta) W^{-\frac12}$ is the inverse of the Birman-Schwinger operator. This allows us to derive Lemma \ref{clr} from classical inequalities about negative eigenvalues of Schr\"odinger operators.

\begin{proof}
 We denote by $N(\mu,W^{-1/2} (-\Delta+\tau) W^{-1/2})$ the number of eigenvalues (counting multiplicities) less than $\mu$ of the operator $W^{-1/2} (-\Delta+\tau) W^{-1/2}$. We shall show that
\begin{equation}
 \label{eq:clr}
N(\mu,W^{-1/2} (-\Delta) W^{-1/2}) \leq C_N'\ \mu^{N/2} \int_\Omega W^{N/2} \,dx
\end{equation}
for $N\geq 3$ and that
\begin{equation}
 \label{eq:clr12}
N(\mu,W^{-1/2} (-\Delta+\tau) W^{-1/2}) \leq C_{N,p}'\ \mu^{p}\ \tau^{-p+\frac N2} \int_\Omega W^p \,dx
\end{equation}
for $N=1,2$ and $p$ as stated in the lemma. Obviously, these bounds are equivalent to those stated in the lemma.

To prove \eqref{eq:clr} for $N\geq 3$ we note that $N(\mu,W^{-\frac12} (-\Delta) W^{-\frac12})$ is equal to the number of eigenvalues greater than $1/\mu$ of the operator $W^{\frac12} (-\Delta)^{-1} W^{\frac12}$. By the Birman-Schwinger principle this number is equal to the number of negative eigenvalues of the Schr\"odinger operator $-\Delta - \mu W$. Hence \eqref{eq:clr} is just a restatement of the Cwikel-Lieb-Rozenblum inequality \cite{LaWe,Hu}.

In order to prove \eqref{eq:clr12} for $N=1,2$, we use an inequality of Lieb and Thirring \cite{LiTh}, which states that for any non-negative operators $A$ and $B$ and for any $p\geq 1$, one has $\tr(AB^2A)^p \leq \tr A^{2p}B^{2p}$. For us, this implies that
$$
N(\mu,W^{-\frac12} (-\Delta+\tau) W^{-\frac12}) 
\leq \mu^{p} \tr \left( W^{\frac12} (-\Delta+\tau)^{-1} W^{\frac12} \right)^p
\leq \mu^{p} \tr W^p (-\Delta+\tau)^{-p}
$$
for $p\geq 1$. Now we use the fact that the integral kernel of $(-\Delta+\tau)^{-p}$, where $-\Delta$ is the Dirichlet Laplacian, is pointwise bounded by the same integral kernel, but now with $-\Delta$ being the Laplacian on $\R^N$. (This is true for the integral kernel of the semi-group $\exp(\beta\Delta)$ by the maximum principle, and follows for $(-\Delta+\tau)^{-p}$ by integration against $e^{-\beta\tau}\beta^{p-1}d\beta$.) Hence, we can bound
$$
\tr W^p (-\Delta+\tau)^{-p} 
\leq \int_\Omega W^p \,dx \ \frac{1}{(2\pi)^N}\int_{\R^N} \frac{d\xi}{(\xi^2+\tau)^p}
= C_{N,p} \ \tau^{-p+\frac N2} \int_\Omega W^p \,dx \,.
$$
Here the constant $C_{N,p}$ is finite for any $p\geq 1$ if $N=1$ and for any $p>1$ if $N=2$. This proves \eqref{eq:clr12}.
\end{proof}


\subsection{Proof of Theorem \ref{hlt12}}

In \cite{FrLiSe2} it was shown that Theorem \ref{clrweighted} has the following consequence.

\begin{corollary}\label{ltweighted}
 Assume that
\begin{enumerate}
 \item[(i')]\label{it:gnweighted} $T$ satisfies a Sobolev interpolation inequality with $2<q<\infty$ and $0<\theta<1$, that is, there is a constant $S>0$ such that for all $u\in\mathrm{dom}\, t$,
\begin{equation}\label{eq:gnweighted}
t[u]^{\theta} \|u\|^{2(1-\theta)} \geq S \left( \int_X |u|^q \,dx \right)^{2/q}  \,.
\end{equation}
\end{enumerate}
Moreover, suppose that Assumption \ref{ass:bd} holds with $\kappa$ replaced by $q/(q-2)$. Define $0<\kappa<\infty$ and $0<\gamma<\infty$ by
\begin{equation}\label{eq:gammakappa}
 \gamma= \frac{q(1-\theta)}{q-2}\,, \qquad \kappa=\frac{q \theta}{q-2} \,,
\end{equation}
Then for all $\tilde\gamma>\gamma$ and for all $V\in L^{\tilde\gamma+\kappa}(X)$ the negative eigenvalues $E_j$ of $T+V$ satisfy 
\begin{equation}\label{eq:ltweighted}
\sum_j |E_j|^{\tilde\gamma} \leq L_{\tilde\gamma} \int_X V_-^{\tilde\gamma+\kappa} \,dx
\end{equation}
with
$$
L_{\tilde\gamma} \leq 
\frac{\tilde\gamma^{\tilde\gamma+1}}{\gamma^\gamma (\tilde\gamma-\gamma)^{\tilde\gamma-\gamma}}
\ \frac{\Gamma(\gamma+\kappa+1) \Gamma(\tilde\gamma-\gamma)}{\Gamma(\tilde\gamma+\kappa+1)} 
\ e^{\gamma+\kappa-1} (\theta^{-\theta} (1-\theta)^{-1+\theta} S)^{-\gamma-\kappa} \,.
$$
\end{corollary}

Corollary \ref{ltweighted} implies Theorem \ref{hlt12} in the same way in which Theorem \ref{clrweighted} implies Theorem \ref{hclr}. Assumptions (a) and (b) are clearly satisfied for the quadratic form \eqref{eq:teps}, and Theorem \ref{main12} gives (i') with a constant independent of $\epsilon$. Moreover, since Corollary \ref{ltweighted} follows from Theorem \ref{clrweighted} applied to the operator $T+\tau$ (where $\tau>0$ is an arbitrary parameter), Assumption (c) can be replaced by the analogue of (d) where, however, $\Upsilon$ has to be replace by the operator $\Upsilon^{(\tau)}$ corresponding to the quadratic form $t[u]+\tau \|u\|^2$.

Similarly as in the previous subsection, we verify this condition by showing that the operator $\exp(-\beta H_\epsilon^{(\tau)})$
in the space $L^2(\Omega,dx)$ is Hilbert-Schmidt with $H_\epsilon^{(\tau)}:=W^{-\frac12} (-\Delta-(1-\epsilon)(2D_\Omega)^{-2}+\tau) W^{-\frac12}$. The latter condition is derived as before from the lower bound on the eigenvalues of the operator $W^{-\frac12} (-\epsilon\Delta+\tau) W^{-\frac12}$ stated in Lemma \ref{clr}. This concludes the proof of Theorem \ref{hlt12}.
\qed


\bibliographystyle{amsalpha}

\end{document}